\documentclass{amsart} 
\usepackage{lmodern}
\usepackage[utf8]{inputenc}
\usepackage{latexcolors}  % Load xcolor first with all options
\usepackage{amsmath, amsthm, amssymb, hyperref}
\usepackage[margin=1in]{geometry}
\usepackage{enumerate}
\usepackage[shortlabels]{enumitem}
\usepackage{mathrsfs}
\usepackage{mathtools}
\usepackage{multicol}
\usepackage{tikz-cd}
\usepackage{subcaption}
\usetikzlibrary{backgrounds}
\usetikzlibrary{decorations.pathmorphing}
\usetikzlibrary{decorations.pathreplacing}
\usepackage{graphicx, float} 
\usepackage[capitalise, noabbrev]{cleveref}
\usepackage{venndiagram}
\usepackage{tikz-cd}
\usepackage[normalem]{ulem}
\usepackage{verbatim}
\usetikzlibrary{shapes.arrows,arrows.meta}
\usetikzlibrary{shapes.geometric, calc}
\usetikzlibrary{3d,arrows.meta,positioning}
\hypersetup{colorlinks, linkcolor=red, citecolor=red, urlcolor=red}
\usepackage[dvipsnames]{xcolor}
\usepackage{mathrsfs}

\newcommand{\st}{\colon}

\newcommand{\bv}{{\bf v}}
\newcommand{\bw}{{\bf w}}
\newcommand{\bx}{{\bf x}}
\newcommand{\by}{{\bf y}}
\newcommand{\bz}{{\bf z}}
\newcommand{\ba}{{\bf a}}
\newcommand{\bb}{{\bf b}}

\newcommand{\R}{\mathbb{R}}

\newcommand{\Z}{\mathbb{Z}}
\newcommand{\cF}{\mathcal{F}}

\newcommand{\cR}{\mathcal{R}}

\newcommand{\sC}{\mathscr{C}}

\newcommand{\qand}{\quad \mbox{and} \quad}
\newcommand{\qfor}{\quad \mbox{for} \quad}

\newtheorem{theorem}{Theorem}[section]

\newtheorem{proposition}[theorem]{Proposition}

\newtheorem{lemma}[theorem]{Lemma}
\theoremstyle{definition}

\newtheorem{example}[theorem]{Example}
\newtheorem{remark}[theorem]{Remark}
\newtheorem*{remark*}{Remark}
\newtheorem{question}[theorem]{Question}
\newtheorem{problem}[theorem]{Problem}

\DeclareMathOperator{\Bier}{Bier}
\DeclareMathOperator{\conv}{conv}
\DeclareMathOperator{\relint}{rel int}

\title{Polytopal Bier spheres and nonrealizable central symmetries}

\author[T. Holleben]{Thiago Holleben}
\address[T. Holleben]{Max Planck Institute for Mathematics in the Sciences, Leipzig, Germany}
\email{holleben@mis.mpg.de}

\author[Y. Yang]{Yirong Yang}
\address[Y. Yang]{School of Mathematics, Georgia Institute of Technology, Atlanta, GA, USA}
\email{yyang917@gatech.edu}

\begin{document}

\begin{abstract}
    Bier spheres arise as deleted joins of simplicial complexes with their combinatorial Alexander duals and form one of the largest known families of simplicial spheres. We study centrally symmetric Bier spheres and give a simple criterion for when they cannot arise as boundaries of centrally symmetric polytopes. From this, we obtain a large new family of simplicial polytopes with combinatorial automorphisms that cannot be realized geometrically. Prior to our construction, the Bokowski--Ewald--Kleinschmidt polytope was the only known simplicial example exhibiting these properties. By Smith theory, these polytopes have noncontractible realization spaces. Finally, we establish that every Bier sphere with at most $12$ vertices is polytopal.
\end{abstract}

\maketitle

\section{Introduction}

In 1910, in an attempt to compute every possible combinatorial type of $4$-polytopes, Br\"uckner~\cite{B1910} found a triangulated $3$-sphere that is not the boundary of any $4$-polytope. In other words, Br\"uckner found the first example of a~\emph{nonpolytopal} sphere. Later, in 1988, Kalai~\cite{K1988} showed that in fact, for $d \ge 5$, most triangulated $(d-1)$-spheres are not boundary complexes of $d$-polytopes. This result was then extended by Pfeifle and Ziegler~\cite{PZ2004} in 2004 to the case $d = 4$. For $d \leq 3$, the situation is completely different, as it is known from Steinitz's theorem (see~\cite[Chapter 4]{Z1995}) that every triangulated $2$-sphere is polytopal. However, although nonpolytopal spheres are asymptotically abundant, explicit examples are still difficult to produce.

In 1990, Villarreal~\cite{V1990} studied a well-known construction in graph theory from an algebraic perspective, showing a simple way of generating Cohen--Macaulay simplicial complexes. Independently, in 1992, Bier showed that Villarreal's construction can be used to produce triangulated spheres, which are nowadays called~\emph{Bier spheres}. These spheres arise as the \emph{deleted join} of a simplicial complex and its combinatorial Alexander dual. 
Since the input for this construction is an arbitrary simplicial complex, Bier spheres form a very large class of simplicial spheres. Indeed, in 2005 Bj\"orner et al.~\cite{bjorner2005bier} showed that there are asymptotically more Bier spheres than simplicial polytopes, and hence most Bier spheres must be nonpolytopal. We note, however, that to this day, there is no known example of a nonpolytopal Bier sphere. The best current result in this direction is due to Timotijevi\'c et al.~\cite{TimotijevicZivaljevicJevtic2025}, who showed that every Bier sphere up to $11$ vertices is polytopal. In addition, Bier spheres arising from threshold complexes are polytopal, and all Bier spheres admit star-shaped realizations \cite{JevticTimotijevicZivaljevic2021}.

As observed in ~\cite{bjorner2005bier}, when a simplicial complex is equal to its Alexander dual, its Bier sphere is \emph{centrally symmetric} (cs). If a cs simplicial sphere is combinatorially isomorphic to the boundary of a cs polytope, then we say it is \emph{cs-polytopal}. We are interested in determining the cs-polytopality of cs Bier spheres. Answers to this problem offer insights not only into the quest of finding a nonpolytopal Bier sphere, but also into some more classical problems in discrete geometry. For example, Mani \cite{Mani1971} proved that every $3$-polytope admits a realization where all combinatorial automorphisms can be realized geometrically. This led to the conjecture that the same property holds for higher-dimensional polytopes. The first and, to the best of our knowledge, only simplicial counterexample to this conjecture is the Bokowski--Ewald--Kleinschmidt (BEK) polytope~\cite{bokowski1984combinatorial}. As we discuss later, a nonrealizable symmetry also implies a noncontractible \emph{realization space}. Much as with nonpolytopal simplicial spheres, whose existence is evident asymptotically but for which explicit examples remain sporadic, there is a gap between general existence and explicit constructions of simplicial polytopes with noncontractible realization spaces. Indeed, despite various universality theorems (see for example \cite{M1988, RG1996,AP2017}), the only explicit simplicial examples are the BEK polytope and a few rather high ($61$, $374$, and $693$) dimensional polytopes (see \cite{Sturmfels1988Isotopy} and \cite[Section 6.2]{BS1989}). In this paper, we give new, accessible constructions that help fill this gap in the literature.

Our first main result, proved in~\cref{s:csnonpolytopal}, is an application of a centrally symmetric version of Gale diagrams from~\cite{mcmullen1968diagrams}. It shows that Bier spheres provide an easy way to construct cs simplicial spheres that admit no geometric cs realization.
\begin{theorem}\label{t:intro1}
    For every even $d \geq 6$ and $2 \leq k \leq d/2$, there exist centrally symmetric $k$-neighborly Bier $(d-2)$-spheres on $2d$ vertices that are not cs-polytopal.
\end{theorem} 
The precise definition of cs-neighborliness can be found in \cref{s:csnonpolytopal}. Roughly speaking, high neighborliness implies high density of the lower-dimensional skeleta of the spheres, imposing more rigid combinatorial structures. Our construction gives an explicit and easy-to-check combinatorial condition on a simplicial complex $\Delta$ that guarantees $\Bier(\Delta)$ to be non-cs-polytopal, see~\cref{prop:noncspoly}.

When a polytope $P$ admits a combinatorial symmetry $\phi$ of order $n$, we may apply the map $\phi$ to the vertex labels of $P$ in order to obtain new realizations of $P$. More concretely, a combinatorial symmetry of order $n$ induces an action of $\Z_n$ on the~\emph{realization space} $\cR(P)$ of $P$, which is the space of all realizations of $P$ modulo affine transformations. A fixed point of this $\Z_n$-action corresponds to a realization in which $\phi$ is represented by an affine symmetry. Therefore, a nonrealizable combinatorial symmetry is an obstruction to having fixed points. The topology of the space $\cR(P)$ was extensively studied by many, including Mn\"ev~\cite{M1988} and Richter-Gebert~\cite{RG1996}, who showed that $\cR(P)$ can have arbitrary homotopy type when $P$ has few vertices or when it is $4$-dimensional (see also~\cite{V2023}). Classical Smith theory relates the obstruction to fixed points to the topology of the space (see~\cref{l:topology,r:topology}). Combining this with our construction, we obtain several new cs-neighborly $d$-polytopes in dimension $d \geq 5$ with noncontractible realization spaces.

\begin{theorem}\label{t:intro2}
    There exist cs-neighborly Bier spheres that are polytopal but not cs-polytopal, and cs-neighborly $(d-1)$-spheres satisfying this property for all $d \ge 5$. In particular, their polytopal realizations have noncontractible realization spaces.
\end{theorem}

\begin{remark}
    The second part of the statement follows from the fact that if a polytope $P$ admits a combinatorial central symmetry that is not realizable by any realization of $P$, then its realization space is not contractible. It has been mentioned several times in the literature (see for example \cite{BG1990, Z1993}) that Mn\"ev observed that nonrealizable combinatorial symmetries are obstructions to the contractibility of realization spaces, using Smith theory. However, we could not find either an explicit statement or a proof of it. We thus provide an exposition of the application of Smith theory to this problem in \cref{subsec:noncontractible}, and include one proof for the case of combinatorial central symmetry, see~\cref{p:topology}. 
\end{remark}

\cref{t:intro2} follows from the case when the spheres from \cref{t:intro1} are polytopal. We note that when deciding the polytopality of the spheres from \cref{t:intro1}, either outcome is significant: if such a sphere is polytopal, then it is a new example of a simplicial polytope with a noncontractible realization space; if it is nonpolytopal, then it will be the first explicit nonpolytopal Bier sphere. In trying to find examples of nonpolytopal Bier spheres, we obtain the following strengthening of the main result from~\cite{JevticTimotijevicZivaljevic2025}. 
\begin{theorem}\label{thm:all12vertBier}
    All Bier spheres with at most $12$ vertices are polytopal.
\end{theorem}
This computational result was obtained using generative AI and independently verified by the authors in SageMath. For details on AI usage and the verification process, see \cref{s:ai}. As a consequence, all Bier spheres from \cref{t:intro1} when $d = 6$ are polytopal. The algorithm also sampled $2098$ polytopal cases for when $d > 6$. These spheres and their repeated suspensions give rise to examples for \cref{t:intro2}. To the best of our knowledge, this is the simplest and most general method of generating simplicial polytopes with nonrealizable combinatorial symmetries (and thus noncontractible realization spaces). Finally, in~\cref{s:futurework} we gather questions that follow from our work.

\section{Preliminaries}
\subsection{Simplicial spheres and polytopes} A~\emph{(abstract) simplicial complex} $\Delta$ on $[n]$ is a collection of subsets of $[n]$ closed under inclusion. Elements of $\Delta$ are called~\emph{faces}, and maximal faces are called~\emph{facets}. For a face $\sigma \in \Delta$, the~\emph{dimension} of $\sigma$ is $\dim \sigma = |\sigma| - 1$. We say $\Delta$ is~\emph{pure} if every facet has the same dimension. Faces of dimension $0$ are called \emph{vertices}, $1$-dimensional faces are called \emph{edges}, and if $\Delta$ is pure of dimension $d$, then the $(d-1)$-faces are called~\emph{ridges}. The~\emph{$f$-vector} of a $d$-dimensional complex $\Delta$ is the sequence of numbers $(f_{-1}, f_0, \dots, f_d)$, where $f_{i}$ is the number of $i$-faces of $\Delta$. A~\emph{geometric simplex} of dimension $d$ in $\R^m$ is the convex hull of a collection of $(d+1)$ affinely independent points in $\R^m$. A~\emph{geometric simplicial complex} is a collection $\Gamma$ of geometric simplices that is closed under inclusion such that if $\sigma_1, \sigma_2 \in \Gamma$, then $\sigma_1 \cap \sigma_2 \in \Gamma$. A geometric simplicial complex $\Gamma$ naturally defines an abstract simplicial complex $\Delta$ and, in this case, we say $\Gamma$ is a~\emph{geometric realization of $\Delta$}. Conversely, every abstract simplicial complex $\Delta$ has a geometric realization in $\R^m$ for some $m$. We write $|\Delta|$ to denote a geometric realization of $\Delta$, and we refer the reader to~\cite[Chapter 1]{M2003} for more details on geometric realizations. If the geometric realization of a simplicial complex $\Delta$ is homeomorphic to a $(d-1)$-dimensional sphere, we say $\Delta$ is a~\emph{simplicial $(d-1)$-sphere}. A simplicial complex $\Delta$ is~\emph{centrally symmetric} or~\emph{cs} if there exists an involution $\iota$ on the vertices of $\Delta$ such that $\iota(\sigma) := \{\iota(v) \st v \in \sigma\} \neq \sigma$ for every nonempty face $\sigma$ of $\Delta$ and $\iota(\iota(\sigma)) = \sigma$ for every $\sigma$. We call the face $\iota(\sigma)$ the~\emph{antipodal face} of $\sigma$. 

A~\emph{polytope} $P \subset \R^m$ is the convex hull of a finite set of points in $\R^m$. A~\emph{supporting hyperplane} of $P$ is a hyperplane $H = \{\bx \in \R^m  \st \ba^\top \bx = b\}$ with $\ba \in \R^m \setminus \{\mathbf{0}\}$ and $b \in \R$ such that $P \cap H \ne \emptyset$, and $P \subseteq \{\bx \in \R^m \st \ba^\top \bx \leq b\}$ or $P \subseteq \{\bx \in \R^m \st \ba^\top \bx \geq b\}$. A set $F \subset \R^m$ is a~\emph{face} of $P$ if there exists a supporting hyperplane $H$ such that $F = P \cap H$. By convention we also consider $\emptyset$ and $P$ to be faces of $P$. The dimension of a face is the dimension of its affine hull. A $d$-dimensional polytope $P$ is~\emph{simplicial} if every $(d-1)$-face is a geometric simplex. In this case, the boundary $\partial P$ of $P$ defines a simplicial $(d-1)$-sphere. If a simplicial sphere $\Delta$ arises as the boundary complex of a simplicial polytope $P$, then we say $\Delta$ is~\emph{polytopal}. The coordinates of the vertices of $P$ are called a \emph{convex realization} of $\Delta$. The first example of a nonpolytopal sphere was found by Br\"uckner~\cite{B1910}. In fact, it is well-known (see e.g.~\cite{K1988,PZ2004,NevoManyOddSpheres}) that most simplicial $(d-1)$-spheres (for $d \geq 4$) are not polytopal.

A set $X \subset \R^m$ is~\emph{centrally symmetric} or~\emph{cs} if $X = - X := \{-\bx \st \bx \in X\}$. In this paper, we focus on the case where $X$ is either a polytope, or the set of vertices of a polytope. Note that if $P$ is a simplicial cs polytope, then the boundary of $P$ defines a cs simplicial sphere. If $\Delta$ is a cs sphere that admits a cs convex realization, then we say $\Delta$ is~\emph{cs-polytopal}. The history and examples of polytopal cs spheres that are not cs-polytopal are the main topic of~\cref{s:csnonpolytopal}.

\subsection{Bier spheres} Given a simplicial complex $\Delta$ on vertex set $V$, the~\emph{Alexander dual} of $\Delta$ is the simplicial complex $\Delta^\ast = \{V \setminus \sigma \st \sigma \not \in \Delta\}$. In order to define Bier spheres, we will need two collections of vertices. Let $X_n = \{x_1, \dots, x_n\}$ and $Y_n = \{y_1, \dots, y_n\}$. If $\Delta$ is a simplicial complex on $X_n$, we denote by $\Delta^\ast_Y$ the Alexander dual of $\Delta$, where the faces are taken in $Y_n$ instead of $X_n$. In other words,  
$$
    \Delta^\ast_Y = \{Y_n \setminus \{y_i \st x_i \in \sigma\} \st \sigma \not \in \Delta\}.
$$ 
Moreover, for a face $\tau$ of $\Delta^\ast_Y$, we denote by $\tau_X$ the corresponding face in $\Delta^\ast$.

The~\emph{Bier sphere} of $\Delta$ is the simplicial sphere on vertex set $X_n \cup Y_n$ defined as follows: 
$$
    \Bier(\Delta) = \{\sigma \cup \tau \ \st \ \sigma \in \Delta \qand \tau \in \Delta^\ast_Y, \ \sigma \cap \tau_X = \emptyset\}.
$$
It is not clear from the definition that $\Bier(\Delta)$ is a simplicial sphere. The first proof of this fact is due to Bier in an unpublished manuscript from 1992 and is recorded in \cite[Section 5.6]{M2003}. We note that a very similar construction had already been studied two years earlier by Villarreal~\cite{V1990} in the context of commutative algebra, where it is shown that a related complex is Cohen--Macaulay. Villarreal's construction takes a simplicial complex $\Delta$ and outputs a simplicial complex $w(\Delta)$ which turns out to be a simplicial ball. The boundary of this ball is exactly the sphere $\Bier(\Delta)$. For more details on the history of these topics, we refer the reader to~\cite{CFHNVT2026}.

Since the input of the Bier operation is a simplicial complex, this class of spheres turns out to be very useful for studying interesting behaviors. Many basic properties of $\Bier(\Delta)$ can be read off from $\Delta$ directly. As an example, it is possible to characterize minimal nonfaces of Bier spheres from the minimal nonfaces of $\Delta$ and $\Delta^\ast_Y$. Concretely, $\Bier(\Delta)$ has three types of minimal nonfaces:
\begin{enumerate}
    \item nonfaces of $\Delta$ on the vertex set $X_n$
    \item the missing edges coming from antipodal points: $\{x_i, y_i\}$
    \item nonfaces of $\Delta^\ast_Y$ on the vertex set $Y_n$.
\end{enumerate}
This perspective is often very helpful, and it is used frequently in more algebraic settings (see for example~\cite[Section 3]{M2011}). We can also use the list of nonfaces above to prove that a certain sphere is not the Bier sphere of any simplicial complex. 

Finally, we note that our definition of simplicial complexes does not require that every singleton is a face. If $\Delta$ is a simplicial complex on vertex set $V$ and $v \in V$ is not a face of $\Delta$, we say $v$ is a~\emph{ghost vertex} of $\Delta$. Ghost vertices turn out to play an important role in the theory of Bier spheres, as most of the results about the polytopality of these spheres revolve around the number of vertices. The sphere $\Bier(\Delta)$ has a ghost vertex for every ghost vertex of $\Delta$ and $\Delta^\ast_Y$. Equivalently, $\Bier(\Delta)$ has a ghost vertex for every minimal nonface of size $1$ or $n-1$ of $\Delta$.

\section{Nonrealizable central symmetries}\label{s:csnonpolytopal}
In this section, we present a class of cs Bier spheres that are not cs-polytopal. The motivation for finding such examples is twofold:

\begin{enumerate}
    \item if such spheres are also not polytopal, then they will be the first known examples of nonpolytopal Bier spheres;
    \item if they turn out to be polytopal, then they will be examples of polytopes with combinatorial symmetries that cannot be realized geometrically.
\end{enumerate}
In~\cref{sec:all12vert} we will discuss how to compute the explicit coordinates of the convex realizations when we land in the second scenario.
\subsection{cs Bier spheres}
It follows from the definition that $\Bier(\Delta)$ is cs when $\Delta = \Delta^*$. This is also observed in \cite{bjorner2005bier} with the following equivalent formulation.
\begin{proposition}\cite[Proposition 6.1]{bjorner2005bier}
    Let $\Delta$ be a simplicial complex on $[n]$. If $A \in \Delta \iff [n] \setminus A \notin \Delta$, then $\Bier(\Delta)$ is centrally symmetric.
\end{proposition}
A simplicial complex is \emph{$k$-neighborly} if every $k$-subset of its vertices forms a face. Since no pair of antipodal vertices in a cs complex is connected, the definition of neighborliness for cs complexes is modified as follows. A cs simplicial complex is \emph{cs $k$-neighborly} if every antipode-free set of $k$ vertices forms a face. Note that while an arbitrary cs complex can have non-unique antipodal pairings, the pairings for a cs $k$-neighborly complex when $k \ge 2$ are uniquely determined because they correspond exactly to the nonedges. The following result characterizes all complexes that give rise to cs $k$-neighborly Bier spheres.
\begin{proposition}\cite[Proposition 6.2]{bjorner2005bier}\label{prop:csneighbor}
Let $\Delta$ be a simplicial complex on $[n]$. Let $1 \le k \le \lfloor (n-1)/2 \rfloor$. Then $\Bier(\Delta)$ is a cs $k$-neighborly $(n-2)$-sphere with $2n$ vertices if and only if
\begin{enumerate}
    \item $A \in \Delta \iff [n]\setminus A \notin \Delta$, for all $A \subseteq [n]$,
    \item $B \in \Delta$ for all $B \subseteq [n]$ such that $|B| \le k$ (and thus $C \notin \Delta$ for all $C \subseteq [n]$ such that $|C| \ge n-k$).
\end{enumerate}
\end{proposition}
A cs $(d-1)$-sphere or a cs $d$-polytope is said to be \emph{cs-neighborly} if it is cs $\lfloor d/2\rfloor$-neighborly. Spheres and polytopes that are cs-neighborly are of considerable interest in discrete geometry. They have the densest possible low-dimensional skeleta compatible with central symmetry, and this rigid structure arises naturally in face enumeration and realization problems. We recommend~\cite{Novik2019CentrallySymmetric} for further background. By \cref{prop:csneighbor}, for odd $n$ there is only one cs-neighborly Bier $(n-2)$-sphere with $2n$ vertices, with $\Delta = \{A \subseteq [n]\colon |A| \le \lfloor n/2 \rfloor\}$. On the other hand, for even $n$, one can take any $\Delta$ that contains all sets of size less than $n/2$, and exactly one set from each pair $(A, [n]\setminus A)$ where $A \subseteq [n]$ such that $|A| = n/2$.

\subsection{cs diagrams and non-cs-polytopality}

To construct cs Bier spheres that are not cs-polytopal, we use the notion of cs diagrams introduced in \cite{mcmullen1968diagrams}. These diagrams were developed as an analogous machinery to the Gale diagram representing a convex polytope. Specifically, if $P$ is a cs $d$-polytope with $2n$ vertices, then its faces can be read off from a diagram consisting of a cs set of $2n$ points in $\R^{n-d}$. Therefore, if our cs Bier $(n-2)$-spheres can be realized as cs $(n-1)$-polytopes, then $d = n-1$ and their cs diagrams will contain points in $\R$, which are much simpler to study than the polytopes themselves.

We closely follow \cite{mcmullen1968diagrams} for the definitions and theorems about cs diagrams. Let $X = \{\pm \bx_1, \dots, \pm \bx_n\}$ be a cs set of points in $\R^d$. Let $X_+ = \{\bx_1, \dots, \bx_n\}$. Assume without loss of generality that $X$ spans $\R^d$. Let $L(X_+)$ be the set of $n$-vectors 
\[
\{(\lambda_1, \dots, \lambda_n): \lambda_1\bx_1 + \cdots \lambda_n\bx_n = 0\}.
\]
Observe that $L(X_+)$ is a vector space of dimension $n-d$. Let $\ba_j = (a_{j1}, \dots, a_{jn})$, $j \in [n-d]$ be a basis of $L(X_+)$. Let 
\[
\bar{\bx}_i = (a_{1i}, \dots, a_{n-d, i}) \qfor i \in [n]. 
\]
Then $\overline{X}:= \{\pm\bar{\bx}_1, \dots, \pm\bar{\bx}_n\} \subset \R^{n-d}$ will be called a \emph{cs transform} of $X$. The cs transform is defined uniquely up to linear equivalence. Suppose further that $X$ is the set of vertices of a cs $d$-polytope $P$. Any cs set of $2n$ points in $\R^{n-d}$ combinatorially isomorphic to a cs transform of $X$ is called a \emph{cs diagram} of $P$. 

The next theorem uses the same notation as the previous paragraph. 
\begin{theorem}[{\cite{mcmullen1968diagrams}}]\label{thm:csdiagram}
    Let $\varepsilon_1 \bx_{i_1}, \dots, \varepsilon_r \bx_{i_r} (\varepsilon_\ell = \pm 1)$ be any $r$ points of $X$ with distinct suffixes. Then $\conv \{\varepsilon_1 \bx_{i_1}, \dots, \varepsilon_r \bx_{i_r}\}$ is a face of $P$ if and only if
    \[
    \varepsilon_1 \bar{\bx}_{i_1} + \cdots + \varepsilon_r \bar{\bx}_{i_r} \in \relint \conv\{\delta_1\bar{\bx}_{j_1}+ \cdots + \delta_{n-r} \bar{\bx}_{j_{n-r}} \st \delta_1, \dots, \delta_{n-r} \in \{\pm 1\}\},
    \]
    where $\{j_1, \dots, j_{n-r}\} = [n] \setminus \{i_1, \dots, i_r\}$, taking the right-hand side expression to be $\{0\}$ if $r = n$. 
\end{theorem}
We obtain the following result using this criterion.
\begin{proposition}\label{prop:noncspoly}
    Let $n$ be even and $\Delta$ be a simplicial complex on $[n]$ such that $\Bier(\Delta)$ is a cs $k$-neighborly sphere for some $2 \le k \le (n-2)/2$. Moreover, suppose there exist distinct $r, \ell \in [n]$ and $C, D \in {[n]\setminus \{r, \ell\} \choose (n-2)/2}$ such that 
\begin{enumerate}
    \item $C \sqcup \{r\} \in \Delta = \Delta^* \qand [n] \setminus (C \sqcup \{\ell\}) \in \Delta = \Delta^*$
    \item $D \sqcup \{\ell\} \in \Delta = \Delta^* \qand [n] \setminus (D \sqcup \{r\}) \in \Delta = \Delta^*$
\end{enumerate}
then $\Bier(\Delta)$ is not cs-polytopal. 
\end{proposition}
\begin{proof}
    Suppose for contradiction that $\Bier(\Delta)$ is isomorphic to the boundary of a cs polytope $P$. Denote the vertices of $\Bier(\Delta)$ as $x_1, \dots, x_n, y_1, \dots, y_n$, where $x_i, y_i$ are antipodes for $i \in [n]$. Since it is cs $k$-neighborly, any pair of non-antipodal vertices is connected. Therefore, the antipodes of $\Bier(\Delta)$ must be antipodes in the realization $P$ as well. For each $i \in [n]$, let $\bx_i, \by_i \in \R^{n-1}$ denote the vertices of $P$ corresponding to $x_i, y_i$ respectively. Let $\left\{\bar{\bx}_1, \dots, \bar{\bx}_n, \bar{\by}_1, \dots, \bar{\by}_n \right\} \subseteq \R$ be a cs transform of the vertex set of $P$. 

    Suppose $C = \{i_1, \dots, i_{(n-2)/2}\}$, and $[n] \setminus (C \sqcup \{r, \ell\}) = \{j_1, \dots, j_{(n-2)/2}\}$. By the first part of condition (1) and the definition of Bier spheres, any set of the form $\{z_{i_1}, \dots, z_{i_{(n-2)/2}}, z_r\}$ where each $z_i \in  \{x_i, y_i\}$ is a face of $\Bier (\Delta)$. Choose $z_i \in \{x_i, y_i\}$ such that $\bar{\bz}_i > 0$ for all $i \in [n]$. By \cref{thm:csdiagram}, this implies 
    \[
    \bar{\bz}_{i_1} + \cdots +\bar{\bz}_{i_{(n-2)/2}} + \bar{\bz}_r < \bar{\bz}_{j_1} + \cdots +\bar{\bz}_{j_{(n-2)/2}} + \bar{\bz}_\ell.
    \]
    Similarly, by the second part of condition (1), any set of the form $\{z_{j_1}, \dots, z_{j_{(n-2)/2}}, z_r\}$ where each $z_i \in  \{x_i, y_i\}$ is a face of $\Bier (\Delta)$. Therefore, assuming again all $\bar{\bz}_i > 0$, we obtain
    \[
    \bar{\bz}_{j_1} + \cdots +\bar{\bz}_{j_{(n-2)/2}} + \bar{\bz}_r < \bar{\bz}_{i_1} + \cdots +\bar{\bz}_{i_{(n-2)/2}} + \bar{\bz}_\ell.
    \]
    These two inequalities together imply $\bar{\bz}_r < \bar{\bz}_\ell$. On the other hand, by an analogous argument, (2) implies $\bar{\bz}_\ell < \bar{\bz}_r$. This is a contradiction. Therefore, $\Bier (\Delta)$ cannot be realized as a cs polytope. 
\end{proof}
\begin{example}\label{ex:n=6noncs}
    Let $n = 6$. Let $\Delta$ be the simplicial complex generated by the following faces
    \begin{enumerate}
        \item all sets $A \subseteq [6]$ with $|A| \le 2$
        \item $\{1, 2, 3\}, \{3,5,6\}, \{2,4,6\}, \{1,4,5\}$ (taking $r = 3$, $\ell = 4$, $C = \{1, 2\}$, $D = \{2, 6\}$)
        \item one set from each pair $A, [6]\setminus A$ for $A \in \{\{1,2,5\},\{1,2,6\},\{1,3,4\},\{1,3,6\},\{1,4,6\},\{1,5,6\}\}$. 
    \end{enumerate}
Then $\Delta$ satisfies the hypothesis of \cref{prop:noncspoly}, and hence $\Bier(\Delta)$ is a cs sphere that is not cs-polytopal. 
\end{example}
\begin{remark}
    The smallest $n$ for which $\Delta$ can satisfy \cref{prop:noncspoly} is $n = 6$. When $n = 4$, the $n/2$-sized pairs are $\{\{1,2\},\{3,4\}\}, \{\{1,3\},\{2,4\}\},\{\{1,4\},\{2,3\}\}$. Taking any $C, r, \ell$, the leftover pair must contain the set $\{r, \ell\}$, so we cannot find $D$.
\end{remark}
Note that $r, \ell, C, D$ satisfying the conditions in \cref{prop:noncspoly} exist for all even $n \ge 6$, so \cref{t:intro1} follows from \cref{prop:noncspoly}. As we will see in \cref{sec:all12vert}, all Bier spheres with 12 vertices can be realized as boundaries of polytopes with explicit coordinates. Therefore, \cref{ex:n=6noncs} provides us with constructions of cs-neighborly Bier spheres that are polytopal but not cs-polytopal. In other words, the combinatorial central symmetry on the sphere cannot be realized geometrically on any convex realization of the sphere. Our optimization algorithm also yields some higher-dimensional instances, sampling a total of $2098$ constructions analogous to \cref{ex:n=6noncs} for $n \in \{8, 10, 12, 14, 16, 18\}$ with convex realizations, see \cite{ourcode}.

Given a simplicial complex $\Delta$, the \emph{two-point suspension} of $\Delta$ with two new vertices $v, w$ is the simplicial complex generated by $\{v \cup \sigma, w \cup \sigma \st \sigma \in \Delta\}$. Given a full-dimensional polytope $P \subset \R^m$, the \emph{bipyramid} over $P$ is obtained by embedding $P$ into $\R^{m+1}$, computing the barycenter $\bb$ of its vertices and taking the convex hull of $P \times \{0\} \cup \{(\bb, 1), (\bb, -1)\}$. The following proposition allows us to extend the sphere constructions to higher dimensions. 
\begin{proposition}\label{prop:suspensionnoncs}
Let $S$ be a centrally symmetric $(d-1)$-sphere on $2n$ vertices, and let $S'$ be its two-point suspension. If $S$ is polytopal but not cs-polytopal, then so is $S'$. In addition, if $S$ is cs $k$-neighborly for some $k$, then so is $S'$. 
\end{proposition}
\begin{proof}
Suppose $S$ is realized by some polytope $P$. Then the bipyramid $P'$ of $P$ realizes $S'$. Now, suppose $S'$ is realized by some cs polytope $P'$. Let $v, w$ denote the two vertices introduced by the suspension. Since $v, w$ must be connected to all vertices of $S$, they must be antipodes in $P'$. Let $\bv, \bw$ denote the two corresponding vertices. Let $P$ be the orthogonal projection of $P'$ along the vector $\bv - \bw = 2\bv$. Then $P$ is cs. 

We will show that the boundary of $P$ is isomorphic to $S$. For every face $\sigma$ of $S$, let $F_\sigma$ denote the face of $P'$ corresponding to $\sigma$. The set of facets of $P'$ containing $\bv$ is exactly \{$\conv(F_\sigma, \bv) \st \sigma \text{ is a facet of } S\}$. The outer normal of every facet $\conv(F_\sigma, \bv)$ has positive component along the projection vector $2\bv$. Moreover, none of these facets is parallel to $\bv\bw$, otherwise both vertices would be in that facet. By convexity of $P'$, we can conclude that the projection is a simplicial bijection between the boundary of $\bigcup_{\sigma \in S}\conv(F_\sigma, \bv)$, which is isomorphic to $S$, and the boundary of $P$. 
    
The last part of the proposition follows immediately from the definition of cs $k$-neighborliness and suspension. 
\end{proof}

\subsection{The Bokowski--Ewald--Kleinschmidt example} To the best of our knowledge, prior to our construction, there was only one example of a polytopal but non-cs-polytopal cs simplicial sphere, modulo any extension of the example using \cref{prop:suspensionnoncs}. This example is constructed in~\cite{bokowski1984combinatorial} and is denoted by $\sC$. The sphere $\sC$ is $3$-dimensional with $f$-vector $(1, 10, 38, 56, 28)$, and its facets are listed in \cref{tab:oldexfacets}. At the end of \cite[Section 2]{bokowski1984combinatorial}, it is remarked that the central symmetry of $\sC$ given by the involution 
\[
\varphi(1) = 10, \quad \varphi(2) = 5, \quad \varphi(3) = 6, \quad \varphi(4) = 9, \quad \varphi(7) = 8
\]
is not induced by an affine symmetry of any polytope realizing $\sC$. In \cite[Section 6.2]{BS1989}, it is mentioned that the partial chirotope uniquely determined by $\sC$ already violates the symmetry. Here we present an alternative proof of the non-cs-polytopality of $\sC$ using \cref{thm:csdiagram}.

\begin{table}[ht]
\centering
\begin{tabular}{cccc @{\hspace{3em}} cccc @{\hspace{3em}} cccc @{\hspace{3em}} cccc}
\hline
\noalign{\vskip 4pt}
1 & 2 & 3  & 7 & 1 & 2 & 4 & 8 & 9  & 2  & 3  & 8 & 9 & 2  & 10 & 7  \\
1 & 3 & 4  & 7 & 2 & 3 & 4 & 8 & 9  & 3  & 10 & 8 & 2 & 3  & 10 & 7  \\
1 & 4 & 6  & 7 & 1 & 5 & 6 & 8 & 9  & 10 & 6  & 8 & 9 & 5  & 6  & 7  \\
4 & 5 & 6  & 7 & 1 & 4 & 5 & 8 & 10 & 5  & 6  & 8 & 9 & 10 & 5  & 7  \\
1 & 2 & 6  & 7 & 1 & 2 & 6 & 8 & 9  & 2  & 6  & 8 & 9 & 2  & 6  & 7  \\
3 & 4 & 5  & 7 & 3 & 4 & 5 & 8 & 3  & 10 & 5  & 8 & 3 & 10 & 5  & 7  \\
1 & 4 & 5  & 6 & 1 & 2 & 3 & 4 & 9  & 10 & 5  & 6 & 9 & 2  & 3  & 10 \\
\noalign{\vskip 4pt}
\hline
\end{tabular}
\caption{The list of facets of the polytopal but non-cs-polytopal cs sphere $\sC$. \label{tab:oldexfacets}}
\end{table}

\begin{proposition}
    $\sC$ is not cs-polytopal.
\end{proposition}
\begin{proof}
Suppose for contradiction that $\sC$ is isomorphic to the boundary of a cs polytope $P$. Denote the vertices of $\sC$ as \[
  x_1=1,\,y_1=10;\quad x_2=2,\,y_2=5;\quad x_3=3,\,y_3=6;\quad
  x_4=4,\,y_4=9;\quad x_5=7,\,y_5=8.
\]
One can check that these antipodal pairings are the only possible central symmetry of $\sC$. Although $\sC$ misses two edges between the non-antipodal pairs of vertices $\{1, 9\}, \{4, 10\}$, making them antipodal does not yield a free involution. For each $i \in [5]$, let $\bx_i, \by_i \in \R^4$ denote the vertices of $P$ corresponding to $x_i, y_i$ respectively. Let $\left\{\bar{\bx}_1, \dots, \bar{\bx}_5, \bar{\by}_1 = -\bar{\bx}_1, \dots, \bar{\by}_5 = -\bar{\bx}_5 \right\} \subseteq \R$ be a cs transform of the vertex set of $P$. We first show that if $\{z_{i_1}, \dots, z_{i_4}\}$ is a nonface and $\bar{\bz}_{i_1} + \cdots +\bar{\bz}_{i_4}-\bar{\bz}_{i_5} \ne 0$, then $\bar{\bz}_{i_1} + \cdots +\bar{\bz}_{i_4}-\bar{\bz}_{i_5}$ must have the same sign as $\bar{\bz}_{i_1} + \cdots +\bar{\bz}_{i_4}+\bar{\bz}_{i_5}$. Suppose $\bar{\bz}_{i_1} + \cdots +\bar{\bz}_{i_4} > \bar{\bz}_{i_5}$. If $\bar{\bz}_{i_5} > 0$, then $\bar{\bz}_{i_1} + \cdots +\bar{\bz}_{i_4} > \bar{\bz}_{i_5} > - \bar{\bz}_{i_5}$. If $\bar{\bz}_{i_5} < 0$, then it must be that $\bar{\bz}_{i_1} + \cdots +\bar{\bz}_{i_4} > - \bar{\bz}_{i_5}$, since \cref{thm:csdiagram} forbids $\bar{\bz}_{i_1} + \cdots +\bar{\bz}_{i_4} \in \relint \conv\{\pm \bar{\bz}_{i_5}\}$. The argument for when $\bar{\bz}_{i_1} + \cdots +\bar{\bz}_{i_4} < \bar{\bz}_{i_5}$ is analogous. 

We apply the above observation consecutively to the nonfaces
\[
  \{y_1,y_2,y_4,y_5\}=\{10, 5, 9, 8\},\quad
  \{y_1,y_2,x_3,y_4\}=\{10, 5, 3, 9\},\quad
  \{y_1,x_3,y_4,x_5\}=\{10, 6, 9, 7\}.
\]
The result is that 
\[\bar{\by}_1 + \bar{\by}_2 + \bar{\by}_4 + \bar{\by}_5 - \bar{\bx}_3 = \bar{\by}_1 + \bar{\by}_2 + \bar{\by}_3 +\bar{\by}_4 - \bar{\bx}_5
\]
must have the same sign as 
\[
\bar{\by}_1 + \bar{\bx}_3 + \bar{\by}_4 + \bar{\bx}_5 + \bar{\bx}_2 = \bar{\by}_1 + \bar{\bx}_2 + \bar{\bx}_3 + \bar{\by}_4 + \bar{\bx}_5.
\]
However, $\{y_1, y_2, y_3, y_4\}= \{10, 5, 6, 9\}$ and $\{y_1, x_2,  x_3, y_4\} = \{10, 2, 3, 9\}$ are both facets. By \cref{thm:csdiagram}, if $\bar{\bx}_5 > 0$, then $\bar{\by}_1 + \bar{\by}_2 + \bar{\by}_3 +\bar{\by}_4 - \bar{\bx}_5 < 0$ and $\bar{\by}_1 + \bar{\bx}_2 + \bar{\bx}_3 + \bar{\by}_4 + \bar{\bx}_5 > 0$. If $\bar{\bx}_5 < 0$, then $\bar{\by}_1 + \bar{\by}_2 + \bar{\by}_3 +\bar{\by}_4 - \bar{\bx}_5 > 0$ and $\bar{\by}_1 + \bar{\bx}_2 + \bar{\bx}_3 + \bar{\by}_4 + \bar{\bx}_5 < 0$. This is a contradiction.
\end{proof}

The polytopality of the BEK sphere $\sC$ is established in \cite{bokowski1984combinatorial} and later proved independently in \cite{altshuler1984remark} and \cite[Section 3.4]{firsching2017realizability}. We remark that $\sC$ is not cs $k$-neighborly for any $k \ge 2$ because of its nonedges $\{1, 9\}, \{4, 10\}$. Therefore, \cref{prop:noncspoly} provides the first construction of cs-neighborly spheres that are polytopal but not cs-polytopal. Using \cref{prop:suspensionnoncs}, we can conclude \cref{t:intro2} from \cref{ex:n=6noncs} and the BEK sphere. 

\begin{remark}
As mentioned in \cite{BG1990}, there are
errors in both explicit convex realizations of $\sC$ provided in \cite{bokowski1984combinatorial}. These errors do not affect any result of that paper. In \cite[Table 2]{bokowski1984combinatorial}, the third coordinate of vertex $1$, denoted $x^1_3$, should be $1/3$ instead of $1$. At the end of the paper, one can find a different realization with vertices $a$ through $j$. A correction of the coordinates of these vertices can be found in \cite{Bokowski1991GeometricFlatEmbedding}. 
\end{remark}

\subsection{Noncontractible realization spaces of polytopes}\label{subsec:noncontractible} Let $S$ be a polytopal $(d-1)$-sphere on $n$ vertices, such that $S$ is combinatorially isomorphic to the boundary complex of a polytope $P \subset \R^d$. Let $R(S)$ denote the set of all possible convex realizations of $P$ inside $\R^d$, that is, $R(S) = \{(\ba_1, \dots, \ba_{n}) \in \R^d \st \conv(\ba_1, \dots, \ba_n) \cong P\}$. Note that the group of affine transformations acts on $R(S)$ by simply applying the affine transformation to a given realization of $P$. The~\emph{realization space} $\cR(S)$ of $S$ (or equivalently $\cR(P)$ of $P$) is the set of orbits of $R(S)$ under the group action mentioned above. A lot of research has been done in order to better understand the topological properties of $\cR(S)$. As an example, when $d = 3$, Steinitz's theorem (see~\cite[Chapter 4]{Z1995}) implies that $\cR(S)$ is contractible for every simplicial $2$-sphere $S$. 
For $d > 3$, the situation becomes drastically different, as can be seen from universality statements such as Mn{\"e}v's theorem~\cite{M1988} for polytopes with few vertices, Richter-Gebert's theorem~\cite{RG1996} for the case $d = 4$, and the Adiprasito--Padrol theorem for neighborly polytopes~\cite{AP2017} (see also~\cite{blog,V2023}). These theorems imply that the realization spaces of polytopes can be ``as bad as possible" in a very precise sense, see~\cite{V2006}.

    While there are many statements about the possible behaviors of realization spaces, in practice, most polytopes that have been studied have contractible realization spaces. In fact, in~\cite{Z1993} Ziegler mentions that the unique known simplicial polytope with a noncontractible realization space is the $3$-sphere $\sC$ originally defined in~\cite{bokowski1984combinatorial} that we considered earlier in this section. The realization space of $\sC$ is even disconnected~\cite{BG1990}.

    One way to understand why the realization space of the sphere $\sC$ fails to be contractible is by considering the fact that $\sC$ admits a combinatorial symmetry that is not realized geometrically. This observation follows from classical results on Smith theory. Although we believe~\cref{p:topology} is known to experts (see Ziegler's remark on Mn\"ev's paper~\cite{M1988} in~\cite[p. 4]{Z1993}), we include the following statements here for the reader's convenience.

    \begin{lemma}\label{l:eigenvalues}
        Let $S$ be a cs $(d-1)$-sphere that is polytopal but not cs-polytopal. Then, for every polytope $P \subset \R^d$ with boundary complex $S$, the combinatorial central symmetry of $P$ cannot be realized by affine transformations of $\R^d$.
    \end{lemma}

    \begin{proof}
        Assume for the sake of contradiction that there is a simplicial polytope $P \subset \R^d$ with boundary complex $S$, where the combinatorial symmetry of $S$ is realized by an affine transformation $A$, and $P$ is not cs.
            
        Since $P$ is compact and convex, and $A(P) = P$, Brouwer's fixed point theorem implies $A$ has a fixed point $\bx$, which we may view as the origin after a translation. By definition, applying translations does not change
        the point in $\cR(S)$, so we may assume without loss of generality that $P$ contains the origin, and that $A$ is a linear transformation.
        
        The linear transformation $A$ is an involution, so we know $A^2 = I$, and hence the minimal polynomial of $A$ divides $x^2 - 1 = (x-1)(x+1)$. The roots of this polynomial imply that all eigenvalues of $A$ are in $\{-1, 1\}$. If $A$ had an eigenvalue equal to $1$, then taking an eigenvector $\bv$ of $A$ such that $A\bv = \bv$, we get a line passing through the origin that is pointwise fixed by $A$. Since $P$ contains the origin, this line intersects the boundary of $P$ and hence there is a point $\by$ in the boundary of $P$ such that $A \by = \by$. The point $\by$ is in the relative interior of a unique face $\sigma$ of $P$. Since $A$ maps faces to faces, this means that $A$ fixes the entire face $\sigma$, so $A$ does not realize the free combinatorial involution of $S$. Therefore, every eigenvalue of $A$ must be equal to $-1$. Diagonalizing $A$, we then conclude that $A = -I$. This is a contradiction, since this implies $P$ is cs, but we know $S$ is not cs-polytopal.
    \end{proof}

    \begin{remark}
        We note that~\cref{l:eigenvalues} is required only for the case of central symmetries, since the definition of a cs polytope requires that a specific linear transformation realize the geometric symmetry. In general, one could say a combinatorial symmetry $\phi$ of a polytope $P$ is geometrically realized by a realization $Q$ of $P$ whenever $Q$ is a fixed point of the map induced by $\phi$ on $\cR(P)$.
    \end{remark}

    \begin{lemma}[{\cite[Corollary 4.6]{B1960}}]\label{l:topology}
        Let $S$ be a polytopal sphere such that $\cR(S)$ is contractible. For a prime $p$, suppose that $\Z_p$ acts on $\cR(S)$.  Then the action of $\Z_p$ on $\cR(S)$ has a fixed point.
    \end{lemma}

    \begin{remark}\label{r:topology}
        \cref{l:topology} is a very special case of the actual statement of~\cite[Corollary 4.6]{B1960}. The exact assumptions of the statement are very technical, so we avoid the explicit definitions here, but we note that most of the terms are defined in the references we mention. We now briefly explain why the assumptions in~\cite[Corollary 4.6]{B1960} are satisfied in our setting:
        \begin{enumerate}
            \item For every polytopal sphere $S$, the realization space $\cR(S)$ is an open semialgebraic set, and in particular $\cR(S)$ is locally compact Hausdorff.
            \item Semialgebraic sets  are paracompact topological spaces (see~\cite{blog} and~\cite[Chapters 3 and 4]{V2023}).
            \item Since $\cR(S)$ is a semialgebraic set inside $\R^t$ for some finite $t$, the dimension condition holds.
            \item As is mentioned in~\cite[Preliminaries]{B1960}, the notation $H^i(X, \Z_p)$ in~\cite[Corollary 4.6]{B1960} denotes the $i$-th graded piece of the Alexander--Spanier cohomology ring (with closed support). Since $\cR(S)$ satisfies the properties from the items above, it is known (see~\cite[Example 6.9.2 and Corollary 6.9.5]{S1966}) that the graded pieces of this cohomology ring coincide with singular cohomology groups. In particular, since in~\cref{l:topology} the space $\cR(S)$ is contractible, these groups are all trivial (in nonzero cohomology degrees). 
        \end{enumerate}
        
    \end{remark}

    \begin{proposition}\label{p:topology}
        Let $S$ be a cs sphere that is polytopal but not cs-polytopal. Then $\cR(S)$ is not contractible. 
    \end{proposition}

    \begin{proof}
    Suppose for contradiction that $\cR(S)$ is contractible. 
        Since $S$ is a cs sphere, the realization space $\cR(S)$ admits a $\Z_2$ action $\tau : \cR(S) \to \cR(S)$. This action sends a realization $(\ba_1, \dots, \ba_m)$ to $(\ba_{\phi(1)}, \dots, \ba_{\phi(m)})$, where $\phi$ is the central symmetry of $S$. Moreover, by~\cref{l:eigenvalues}, for every $P \in \cR(S)$, we have $\tau(P) \neq P$, otherwise the combinatorial central symmetry of $P$ would be realized by an affine transformation. This is a contradiction since~\cref{l:topology} guarantees that $\tau$ must have a fixed point.
        \end{proof}
    
    A direct consequence of~\cref{p:topology} and~\cref{prop:noncspoly} is that cs Bier spheres can be easily used in order to describe combinatorial scenarios that translate to topological obstructions to a realization space being contractible. 

    \begin{theorem}
        Let $\Delta$ be a simplicial complex satisfying the assumptions in~\cref{prop:noncspoly}. If $\Bier(\Delta)$ is polytopal, then  $\cR(\Bier(\Delta))$ is not contractible.
    \end{theorem}

    We end this section by noting that while~\cref{prop:noncspoly} does not guarantee polytopality, there is no known example of a nonpolytopal Bier sphere. Therefore, if one proves that a Bier sphere satisfies the assumptions in~\cref{prop:noncspoly} but is not polytopal, it would be the first such example.
    
\section{All Bier spheres with $12$ vertices are polytopal}\label{sec:all12vert}
It is shown in \cite{TimotijevicZivaljevicJevtic2025} that all Bier spheres with up to $11$ vertices are polytopal. The method relies on their previous result of realizing all Bier spheres of threshold complexes with explicit coordinates \cite{JevticTimotijevicZivaljevic2021}. Given a non-threshold complex $K$, one can ``approximate" $\Bier(K)$ with $\Bier(L)$, where $L$ is a maximal threshold subcomplex of $K$, through a sequence of bistellar flips. Along each flip from $\Bier(L)$ to $\Bier(K)$, the algorithm attempts to find convex realizations by radially varying the coordinates of the vertices involved in the flip. However, the method becomes insufficient as the number of vertices of the Bier spheres grows past $11$, since there are significantly more non-threshold complexes than threshold ones. The hope was then to find nonpolytopal Bier spheres with $12$ vertices. The same group of authors studied a plausible candidate $\Bier (\mathbb{I}_6)$, where $\mathbb{I}_6$ is the unique $6$-vertex triangulation of the real projective plane, but found it to be polytopal as well \cite{JevticTimotijevicZivaljevic2025}.

In this section, we establish the polytopality of all Bier spheres with $12$ vertices, which have dimensions $4$ through $10$. In addition, we obtain rational realizations for all of these Bier spheres. While the approach from \cite{TimotijevicZivaljevicJevtic2025} rescales vertices along fixed rays, our algorithm deforms the star-shaped realization of each Bier sphere \cite{JevticTimotijevicZivaljevic2021} by letting its vertices move freely as we optimize a smooth function.

\subsection*{Tool and computational resource disclosure, and independent verification}\label{s:ai}
The realization algorithm and its implementation were obtained through iterative interactions with Claude Opus 4.7 and Claude Opus 4.8 (Anthropic), including the choice of the objective function, the star-shaped initialization, the sequence of $\beta$ values used during optimization, and the rounding of the realizations to rational coordinates. We then provided Claude's code to GPT-5.6 Sol (OpenAI), which implemented the same algorithm (with adjustments to the objective function and related parameters) for the sampled higher-dimensional spheres satisfying \cref{prop:noncspoly}. The AI-generated code can be found in the \nolinkurl{ai_generated_code} folder in~\cite{ourcode}.  The key steps of the algorithm are described in Sections \ref{subsec:objective} and \ref{subsec:star-shaped}. We note that the approach of finding a convex realization using \cref{prop:objective} is well established in the literature, see for example \cite{BS1989, BjornerLasVergnasSturmfelsWhiteZiegler1999, firsching2017realizability, TimotijevicZivaljevicJevtic2025}. The AI's contribution lies in the subsequent steps of the algorithm, starting with the formulation of the objective function $L_\beta$ from the feasibility conditions on the matrix minors given by \cref{prop:objective}.

We have verified all outputs independently using exact arithmetic in SageMath 10.3, see the \nolinkurl{polytopal_bier_verification} folder in \cite{ourcode}. The verification steps are specifically as follows.
\begin{itemize}
    \item For each sphere $S$ and the corresponding vertex set $V(S)$ of its realization listed in the JSON files, our code verifies that the facet list of $S$ equals that of $\conv(V(S))$ after relabeling. 
    \item For a simplicial complex $\Delta$ on $[n]$ with ghost vertices $n - h + 1, \dots ,n$ (where the set of ghost vertices is empty when $h = 0$), the induced subcomplex of $\Bier(\Delta)$ on $\{1, \dots, n-h\}$ together with the ghost vertices is exactly $\Delta$. Therefore, for each sphere $S$ in the JSON files, we take the induced subcomplex $\Delta$ as described, and verify that $S = \Bier (\Delta)$. 
    \item Specifically, for each higher-dimensional sphere $S$ analogous to those constructed in \cref{ex:n=6noncs}, our code checks that the defining complex $\Delta$ such that $\Bier(\Delta) = S$ is generated by:
    \begin{itemize}
        \item the $n/2$-faces given by $r, \ell, C, D$ satisfying the conditions of \cref{prop:noncspoly};
        \item exactly one set from each pair $(A, [n]\setminus A)$ for $|A| = n/2$;
        \item all sets $A \subseteq [n]$ such that $|A| = n/2-1$ (this ensures that $S$ is cs-neighborly).
    \end{itemize}
    \item Our code also verifies that the computation covers every Bier sphere with $12$ vertices, and that the sampled higher-dimensional examples of \cref{prop:noncspoly} represent distinct combinatorial types:
    \begin{itemize}
    \item By either computer enumeration or a standard application of Burnside's lemma~\cite{B1955} (which is not due to Burnside~\cite{N1979}) from group theory, one can check that there are $8072$ distinct simplicial complexes on $[6]$ without ghost vertices up to $S_6$-symmetry and Alexander duality. Thus our list of $12$-vertex Bier $4$-spheres is exhaustive at the level of defining complexes (though distinct complexes may give isomorphic Bier spheres). Our code recovers $8072$ nonisomorphic defining complexes.
    
    \item Our result also includes $(4+i)$-dimensional Bier spheres arising from complexes on $[6 + i]$ with $2i$ ghost vertices for $1 \le i \le 6$. The numbers of such complexes for $1 \le i \le 6$ are $180, 20, 5, 2, 1, 1$ respectively, up to $S_{6+i}$-symmetry. This can be obtained by a standard counting argument, and our code recovers $209$ nonisomorphic defining complexes.

    \item Finally, for the higher-dimensional examples of \cref{prop:noncspoly}, the verification code checks that the defining complexes are pairwise nonisomorphic. Since our sampled spheres are all cs-neighborly, the nonedges determine the antipodal pairs of vertices. Therefore, the nonisomorphic defining complexes force their Bier spheres to be pairwise nonisomorphic as well. Thus the number $2098$ counts distinct combinatorial types, rather than different realizations of the same sphere.
\end{itemize}
    
\end{itemize}

\subsection{Objective function}\label{subsec:objective}
We begin with an observation that follows directly from the definitions of a polytope, its faces, and supporting hyperplanes. 
\begin{proposition}\label{prop:objective}
    Let $S$ be a $(d-1)$-sphere on $m$ vertices and let $\cF(S)$ denote its set of facets. Let $F = \{i_1 < \cdots < i_d\} \subseteq [m]$ and $j \in [m] \setminus F$. Moreover, let $V = \{v_1, \dots, v_m\} \subseteq \R^d$. Define
\[
D_V(F, j) := \det \begin{bmatrix}
    1 &\cdots &1 &1\\
    v_{i_1} &\cdots &v_{i_d} &v_j
\end{bmatrix}.
\]
Then the points in $V$ realize $S$ as the boundary of $\conv(V)$ if and only if 
\begin{enumerate}
    \item for $F \in \cF(S)$, the values $\{D_V(F, j): j \notin F\}$ are all nonzero and of the same sign. Equivalently, $D_V(F, j) \cdot D_V(F, j') > 0$ for all $F \in \cF(S)$ and distinct $j, j' \notin F$. 
    \item for every $F \in {[m] \choose d} \setminus \cF(S)$, the values $\{D_V(F, j): j \notin F\}$ are not of the same sign when they are nonzero (if the points in $V$ are further assumed to be in general position, then the values are always nonzero). 
\end{enumerate}
\end{proposition}
Note that condition (2) is redundant for the ``if" direction. Indeed, if condition (1) holds, then $S$ is a subcomplex of $\partial \conv(V)$, which then implies $S = \partial \conv(V)$ since they are simplicial spheres of the same dimension. However, during optimization, condition (1) holds only approximately until the very end, so condition (2) is imposed to nudge the solutions toward the right direction. 

We set up a smooth objective function using the \emph{softplus} function $s(x) = \ln (1+e^x)$. For $\beta , \lambda, \tau> 0$ define
\[
L_\beta(V) := \sum_{F \in \cF(S)} \sum_{j, j' \notin F, j < j'} s(-\beta D_V(F, j) D_V(F, j')) + \lambda \sum_{F \notin \cF(S)} s(-\beta/\tau \ln\sum_{j, j' \notin F, j < j'} e^{-\tau D_V(F, j)D_V(F, j')}).
\]
Roughly speaking, minimizing $L_\beta$ produces a candidate realization because the first summand $\Sigma_1$ enforces condition (1), and the second summand $\Sigma_2$ enforces condition (2). First, suppose $V$ satisfies both conditions with points in general position. Then so does $tV$ for any $t \ne 0$. When $V$ gets scaled to $tV$, every determinant $D_V(F, j)$ multiplies by $t^d$. Therefore, as $t \to \infty$, the products $D_{tV}(F, j) D_{tV}(F, j')$ in $\Sigma_1$ tend to $+\infty$, so $\Sigma_1(tV) \to 0$. Meanwhile, for every nonface $F$, at least one product $D_{tV}(F, j) D_{tV}(F, j') < 0$, so $-\beta/\tau \ln\sum_{j, j' \notin F, j < j'} e^{-\tau D_{tV}(F, j)D_{tV}(F, j')} \to -\infty$. This means $\Sigma_2(tV) \to 0$ as well. On the other hand, if $V$ violates condition (1), then $-\beta D_V(F, j) D_V(F, j') \ge 0$ for some $F, j, j'$, so 
\[
\Sigma_1 \ge s(-\beta D_V(F, j) D_V(F, j')) \ge s(0) = \ln2.
\]
If $V$ violates condition (2) for some $F \notin \cF(S)$, then all ${m-d \choose 2}$ pairs of $(D_V(F, j), D_V(F, j'))$ have the same sign or contain at least one zero, so $D_V(F, j) D_V(F, j') \ge 0$. This means 
\[
\Sigma_2 \ge \lambda s\left(-\beta/\tau \ln \left({m-d \choose 2} e^0\right)\right).
\]
Therefore, when $L_\beta(V)$ is sufficiently small (for example, $L_\beta(V)<\min \left\{\ln 2, \lambda s\left(-\beta/\tau\ln {m-d \choose 2}\right)\right\}$), $V$ must satisfy both conditions.

The numerical errors will later be resolved with rational rounding and checking with exact arithmetic. The constant $\lambda$ is used to balance the weights of the two summands and is set to $1$ throughout, and $\tau$ determines the sharpness of $-1/\tau \ln\sum_{j, j' \notin F, j<j'} e^{-\tau D_V(F, j)D_V(F, j')}$ and is set to $20$ throughout. On the other hand, $\beta$ controls the sharpness of the entire function and is varied along the sequence $(2, 20, 200)$ and then $(1,3,10,30,100,300)$ when the optimization stalls. At each stage, the objective function is minimized using SciPy’s implementation \cite{Virtanen2020} of the L-BFGS-B algorithm \cite{Zhu1997}, with gradients computed by automatic differentiation in JAX \cite{jax2018github}.

\subsection{Star-shaped initialization}\label{subsec:star-shaped}Initializing the optimization from a random point frequently causes it to fail. The algorithm therefore uses the star-shaped realization constructed in \cite[Section 3]{JevticTimotijevicZivaljevic2021} as the initial point for the minimization. Let $\delta_i := e_i - u/n$, where $u = e_1 + \cdots + e_n$. Consider the orthogonal projection of the vertices $\{\pm e_i: 1 \le i \le n\}$ of the $n$-dimensional crosspolytope onto $H_0 := \{x \in \R^n: \langle u, x\rangle = 0\}$. The resulting vertices are $\{\pm \delta_i: 1 \le i \le n\}$. As shown in \cite[Theorem 3.1]{JevticTimotijevicZivaljevic2021}, for any $(n-2)$-dimensional Bier sphere without ghost vertices, $\{\pm \delta_i: 1 \le i \le n\}$ are exactly the rays of the fan of the star-shaped realization of the Bier sphere. The combinatorial type of an individual sphere is reflected only in which rays span the maximal cones of the fan. Therefore, for our Bier $4$-spheres on $12$ vertices, we initialize at the same coordinates $\begin{bmatrix}
    A &-A
\end{bmatrix}$, where
\begin{align*}
A = &\begin{bmatrix}
\frac{1}{\sqrt2} & -\frac{1}{\sqrt2} & 0 & 0 & 0 & 0 \\[2pt]
\frac{1}{\sqrt6} & \frac{1}{\sqrt6} & -\frac{2}{\sqrt6} & 0 & 0 & 0 \\[2pt]
\frac{1}{\sqrt{12}} & \frac{1}{\sqrt{12}} & \frac{1}{\sqrt{12}} & -\frac{3}{\sqrt{12}} & 0 & 0 \\[2pt]
\frac{1}{\sqrt{20}} & \frac{1}{\sqrt{20}} & \frac{1}{\sqrt{20}} & \frac{1}{\sqrt{20}} & -\frac{4}{\sqrt{20}} & 0 \\[2pt]
\frac{1}{\sqrt{30}} & \frac{1}{\sqrt{30}} & \frac{1}{\sqrt{30}} & \frac{1}{\sqrt{30}} & \frac{1}{\sqrt{30}} & -\frac{5}{\sqrt{30}}
\end{bmatrix}.
\end{align*}
The construction for the higher-dimensional Bier spheres with ghost vertices is similar. Here not every vertex is paired off; a ghost vertex has no antipode, so its ray may be assigned either sign. The seed is then chosen by a search over all sign patterns.  

Once the algorithm obtains a floating point solution, the coordinates are rounded to the nearest rational with a common denominator $D$, where $D$ is the smallest integer in $\{10^2, 3\cdot 10^2, 10^3, 3\cdot 10^3, 10^4, \cdots\}$ such that the rounded rational coordinates pass the exact verification. This verification step is part of the AI-generated implementation. We emphasize that our verification code is entirely independent from it.

\section{Future directions}\label{s:futurework}
Fix $d \ge 4$. Let $s(d, n)$ denote the number of combinatorially distinct pairs $(S, \iota)$ where $S$ is a cs-polytopal $(d-1)$-sphere on $2n$ vertices with antipodal map $\iota$ such that $\iota$ is realized geometrically by some cs convex realization of $S$. Let $\bar{s}(d, n)$ denote the number of combinatorially distinct pairs $(S, \iota)$ where $S$ is a polytopal cs $(d-1)$-sphere on $2n$ vertices with antipodal map $\iota$ that cannot be geometrically realized. We ask whether almost every polytopal cs sphere is not cs-polytopal. More precisely,
\begin{question}
    Is the limit $\lim_{n \to \infty} \frac{s(d,n)}{s(d,n)+\bar{s}(d,n)}$ equal to $0$?
\end{question}

We have seen that Bier spheres give rise to new explicit constructions of simplicial polytopes with noncontractible realization space. It is reasonable to ask whether this nice class of spheres can provide more insights into the universality problem for the realization spaces of polytopes. 
It is known from the results in~\cite{AP2017} (see also~\cite{M1988} and~\cite[Theorem 6.10(ii)]{BS1989}) that a version of the universality theorem holds for arbitrary simplicial polytopes. Our results from~\cref{s:csnonpolytopal} lead us to the following question about the existence of a universality statement for Bier spheres. 
\begin{question}
    Let $V \subset \R^t$ be an open primary basic semialgebraic set defined over $\Z$. Is there a simplicial complex $\Delta$ such that $\Bier(\Delta)$ is a polytopal sphere and $\cR(\Bier(\Delta))$ is stably equivalent to $V$? In particular, given a simplicial complex $\Gamma$, is there a simplicial complex $\Delta$ such that $\cR(\Bier(\Delta))$ has the same homotopy type as $\Gamma$?
\end{question}

It is shown in \cite[Section 6]{bjorner2005bier} that there are at least $\frac{2^{2^n/\sqrt{n}}}{(2n/e)^{2n}}$ combinatorially nonisomorphic $(n-2)$-dimensional Bier spheres on $2n$ vertices. In contrast, there are only $2^{8n^3+O(n^2)}$ combinatorially distinct simplicial polytopes with $2n$ vertices \cite{Goodman, Alon}. This implies that most Bier spheres are not polytopal. However, to this day there is no explicit example of a nonpolytopal Bier sphere, hence the open problem below. As we have just established that all Bier spheres with at most $12$ vertices are polytopal, the search must now start with at least $13$ vertices. 
\begin{problem}\label{prob:nonpolybier}
    Find a nonpolytopal Bier sphere.
\end{problem}

In general, deciding the polytopality of a simplicial sphere is a very hard problem. In fact, already for $3$-spheres this problem is NP-hard \cite{RichterGebertZiegler1995Universal4Polytopes}. The development of different techniques to approach this problem is outlined in \cite[Section 6.4]{RichterGebertZiegler2017OrientedMatroids}. The same reference provides the background for the following discussion. The vertex set of a simplicial polytope defines a uniform oriented matroid, which is also called a \emph{uniform matroid polytope}. The face lattice of the uniform matroid polytope in turn gives rise to a simplicial sphere. The strict containments of the following hierarchy have been proved: \{simplicial polytopes\} $\subsetneq$ \{uniform matroid polytopes\} $\subsetneq$ \{simplicial spheres\}. We have attempted Problem \ref{prob:nonpolybier} by searching for a Bier sphere that is not even realizable by an oriented matroid, but have failed to obtain such an example so far. This leads to the question below. 
\begin{question}
    Are all Bier spheres realizable as oriented matroids?
\end{question}
If the answer is affirmative, then we also obtain an asymptotic lower bound for uniform matroid polytopes of rank $n$ on $2n$ elements, which, to the best of our knowledge, has not been previously studied. Recall that if $F(N,r)$ denotes the number of uniform oriented matroids of rank $r$ on $N$ elements, then for fixed $r$ and $N \geq r \geq 3$ one has
\[
  2^{c_r N^{r-1}} < F(N,r) < 2^{d_r N^{r-1}},
\]
where $c_r,d_r>0$ are constants depending on $r$ \cite[Corollary~7.4.3]{BjornerLasVergnasSturmfelsWhiteZiegler1999}.  This estimate concerns uniform oriented matroids in general, and does not specialize to matroid polytopes. Moreover, in our situation the relevant rank is not fixed but grows with $N$, so these bounds do not immediately say much about the case $r=n$, $N=2n$.

\subsection*{Acknowledgements}
We are grateful to Isabella Novik for helpful feedback and to Martin Winter for insightful conversations. We also thank Julian Pfeifle and Amy Wiebe for discussions on nonrealizability certificates and coding support. Yirong Yang's research was partially supported by NSF MSPRF grant DMS-2602353. 

\bibliographystyle{amsalpha} 
\bibliography{paper_ref}
\end{document}